\documentclass[12pt,titlepage]{amsart} 
\usepackage{amsfonts, amssymb, mathdots, verbatim}
\usepackage{arydshln} 
\usepackage{enumitem}
\usepackage{tikz}
\usepackage{hyperref}
\usepackage{needspace}
\usetikzlibrary{positioning}
\newtheorem{theorem}{Theorem}[section]

\newtheorem{lemma}[theorem]{Lemma}
\newtheorem{corollary}[theorem]{Corollary}
\theoremstyle{remark}
\newtheorem{remark}[theorem]{Remark}

\theoremstyle{definition}

\hypersetup{
  pdfauthor={Anthony Hernandez},
  pdftitle={The Redheffer Matrix Parity Problem }
}

\begin{document}
\title{The Redheffer Matrix Parity Problem}
\author{Anthony Hernandez}
\email{anthony@anthony-hernandez.com}


\begin{abstract} 
A solution determining the parity of unit entries in Redheffer matrices is given. The solution 
solves the problem by enumerating two disjoint sets from which set 
membership determines parity. The solution framework uses generating functions. 
\end{abstract}
\maketitle
\newpage

\section{Introduction}\label{sec:problem}
The divisor matrix $D=(d_{r,s})_{r,s\in\mathbf{N}}$ is defined by $d_{r,s}=1$ if $r$ divides $s$, and is $0$ otherwise.
\footnote{The letters $r$ and $s$ are dummy indices ranging over $\{1,\ldots,n\}$; alphabetical order relative to $n$ is immaterial.}
A finite truncation of the divisor matrix obtained from the upper left $n\times n$ submatrix of $D$, setting each entry in 
the first column to $1$ was first studied by Raymond Redheffer, these are now called Redheffer matrices. We say an 
entry of a Redheffer matrix is a \emph{unit entry} if it equals $1.$ This paper is concerned with solving the following problem:
\begin{quote}
\small
\emph{The Redheffer Matrix Parity Problem -- Determine the numbers $n$ for which the parity of the unit entries in an $n\times n$ Redheffer  matrix is even or odd. }
\end{quote}
To begin solving the problem we should know the number of unit entries, denoted by $U(n)$ in Redheffer matrices. 
$U(n)$ is counted as follows: The first column consists entirely of 1's, contributing $n$ counts. For each subsequent 
column $s \geq 2$, the count of 1's is exactly $d(s)$, the number of divisors of $s$, since an entry $d_{r,s}$ is $1$ precisely 
when $r$ divides $s$. The top-left corner $d_{1,1}$ is counted twice: once in the first column and for a second time on the 
divisor sum --- $d(1)=1.$  Adjusting for this overcounting, we obtain:
\begin{equation}\label{eq:sumdiv}
U(n) = n - 1 + \sum_{s=1}^{n} d(s).
\end{equation}

To determine the parity of $U(n)$, we consider the identity modulo $2$. The divisor function $d(s)$ is odd if and only if $s$ 
is a perfect square, as divisors of non-square integers pair into distinct pairs $(r, s/r)$, contributing an even number to the count. 
Thus, modulo $2$, the sum $\sum_{s=1}^{n} d(s)$ is congruent to the number of perfect squares not exceeding $n$. Let 
$m = \lfloor \sqrt{n} \rfloor$; the perfect squares in $\{1, 2, \dots, n\}$ are $1^2, 2^2, \dots, m^2$. Hence,
$
\sum_{s=1}^{n} d(s) \equiv m \pmod{2}.
$
Substituting back, we obtain: 
\begin{equation*}\label{eq:u(n)mod2}
U(n) \equiv n - 1 + m \pmod{2}.
\end{equation*}
Write $n$ as the decomposition
\begin{equation} \label{eq:offset}
n = m^2 + j
\end{equation}
with $0 \leq j \leq 2m,$ and note that $j$ has the following interpretation: $j$ measures how far above the integer $n$ is 
from $m^2$. In light of that interpretation it is natural to refer to $j$ as the \emph{offset}. Now observe that 
$n \equiv j \pmod{2}$ if $m$ is even, and $n \equiv j+1 \pmod{2}$ if $m$ is odd. A brief case analysis shows that in all cases,
$U(n) \equiv j + 1 \pmod{2}.$ Thus, the parity of $U(n)$ is determined solely by the offset $j$. 
\begin{quote}
\small
\emph{Offset Rule: $U(n)$ is odd when $j$ is even, and even when $j$ is odd.}
\end{quote}
This offset rule immediately suggests a global characterization. Define the sets
\begin{equation}\label{eq:def-JK}
\mathcal{J} = \{ n \in \mathbf{N} : j \text{ is even} \}, \quad \mathcal{K} = \{ n \in \mathbf{N} : j \text{ is odd} \},
\end{equation}
Then $U(n)$ is odd if and only if $n \in \mathcal{J}$, and even if and only if $n \in \mathcal{K}$. The Redheffer Matrix 
Parity Problem thus reduces to enumerating the sets $\mathcal{J}$ and $\mathcal{K}$ explicitly.

\begin{remark}[Scope, and Provenance]\label{rem:scope}
There are known solutions to the Redheffer Matrix Parity Problem. They are short, and elementary treatments. They observe that parity  hinges on the offset from the preceding square and conclude that the answer depends only on whether that offset is even or odd. In those accounts the two offset classes---$\mathcal{J}$ and $\mathcal{K}$---are introduced \emph{by fiat}: they are named at the outset and  used to state the rule, but they are neither derived from a canonical construction nor analyzed as objects in their own right. This paper  makes no claim to novelty for that parity conclusion. What \emph{is} new here is a singular conceptual framework that \emph{constructs} those classes -- $\mathcal{J}$ and $\mathcal {K}$ -- and explains their structure.  In short, the paper replaces ``assume $\mathcal{J}$ and $\mathcal{K}$'' with a principled construction of $\mathcal{J}$ and $\mathcal{K}$, and develops consequences that the brief proofs do not supply. What follows details our solution.
\end{remark}

\section{The Solution Plan}\label{sec:frame}
We begin by constructing a partition of the natural numbers $\mathbf{N}.$ A \emph{partition} of a set is defined to be a collection of nonempty, pairwise disjoint subsets, called blocks, whose union recovers the entire set. One may further \emph{refine} a partition by subdividing one or more of its blocks into smaller nonempty parts so that the resulting collection also forms a partition. It is important for our purposes that the partition of $\mathbf{N}$ can be realized via a bijective map, $f,$ that assigns to each natural number $n$ an ordered pair $(m,j)$ -- with $m$ and $j$ as defined in Equation\eqref{eq:offset}. It will be shown that under $f,$ the first coordinate $m$ determines the block containing $n,$ and the second coordinate $j$ measures how far above $m^{2}$ the number $n$ sits in the block partition. To keep track of the three statistics --- the integer $n$, its block label $m$, and its offset $j$--- we work in the formal power–series ring $\mathbf{C}[[x,z,w]]$ (coefficients in $\mathbf{C}$), where $x,z,w$ are algebraically independent indeterminates, and package them into a generating function,
\begin{equation}\label{eq:F-def}
F(x;z,w)\;=\;\sum_{n=0}^{\infty} x^{\,n}\,z^{\,m}\,w^{\,j}\ \in\ \mathbf{C}[[x,z,w]].
\end{equation}

Our next step is to develop the analytic properties of $F.$ Lemma~\ref{lem:F_not_rational} shows that $F(x;z,w)$ is not a rational function of $x$ over $\mathbf{Q}(z,w)$. We emphasize that global nonrationality does not preclude rational behavior after specializing $(z,w).$ Lemma~\ref{lem:degeneratelocus} identifies an explicit degeneracy locus in the $(z,w)$–plane where $F$ collapses to a rational function. Outside this \emph{degeneracy locus}, the specialization $F(x;z,w)$ is nonrational as a function of $x$; in particular, its coefficient sequence does not satisfy any linear recurrence with constant coefficients. The distinction between Lemma~\ref{lem:degeneratelocus} and Lemma~\ref{lem:F_not_rational} is subtle: the former classifies rational specializations for fixed parameters $(z,w)\in\mathbf{C}^2$, while the latter asserts that the full family $F(x;z,w)$ is nonrational over the parameter field $\mathbf{Q}(z,w)$. In what follows, we work exclusively with nondegenerate specializations -- away from this locus. None of our subsequent arguments relies on rationality; the coefficient extractions we require are routine within each block partition. 

To bring $F$ into a one variable framework we introduce in Section ~\ref{sec:linprojection}, a linear projection $\varphi$ that “reads off’’ the parity of the offset $j$. Working in the standard monomial $\mathbf{C}$-basis $\{1,w,w^{2},\dots\}$ of $\mathbf{C}[w]$, prescribing the values of a $\mathbf{C}$-linear functional $L_w$ on these monomials determines it uniquely. We require it to annihilate squares and encode parity, namely $L_w(1)=0$, $L_w(w^{2r+1})=+1$, and $L_w(w^{2r})=-1$ for $r\ge1$. This pins down a single functional, conveniently realized by evaluation at $0$ and $-1$: $L_w(P)=P(0)-P(-1)$ for $P\in\mathbf{C}[w]$. We then let $\varphi$ act on monomials by $\varphi(x^{n}z^{m}w^{j})=x^{n}\,L_w(w^{j})$ and extend linearly across the series. After specializing $z=1$, the projection produces
\begin{equation}
\varphi\bigl(F(x;1,w)\bigr)=F(x;1,0)-F(x;1,-1)\in\mathbf{C}[[x]],
\label{eq:ogen-specialize}
\end{equation}
a generating function in the single indeterminate $x$, with coefficients denoted $a(n)$. It makes sense to write 
\begin{equation}\label{eq:G-def}
F(x;1,0)-F(x;1,-1)=G(x)=\sum_{n=0}^\infty a(n)x^n.
\end{equation}
This assigns to each $n$ a coefficient $a(n)$ determined solely by the offset $j$ in the decomposition $n=m^2+j$: namely $a(n)=0$ if $j=0$, and $a(n)=(-1)^{\,j-1}$ for $j\ge1$. This one variable form is what enables the classical manipulations used later--- partial sums, and closed forms. It is important to note that we did not introduce $a(n)$ or $G(x)$ by fiat. The linear functional $L_w(P)=P(0)-P(-1)$ is the canonical parity extractor in the $w$--coordinate: it annihilates the square ($j=0$) and maps even/odd offsets to $-1$/$+1$, respectively. Thus $\varphi$ is not an ad hoc maneuver; it is the natural projection to descend from the fully informative trivariate series $F(x;z,w)$ to a one-variable series $G(x)$ whose coefficients are tailored to the parity of the offset. One could, in principle, work entirely within the ring $\mathbf{C}[[x,z,w]]$, but the subsequent manipulations are more standard and efficient in a single variable--hence the projection.

It is a standard result that the generating function for the partial sums of $a(n)$ is given by $H(x)=G(x)/(1-x).$ In Theorem ~\ref{thm:parsum} we determine a closed form for those partial sums in terms of $n.$
\begin{equation}
H(n)=\sum_{k=0}^{n}a(k),
\label{eq:H-def}
\end{equation}
Theorem ~\ref{thm:parsum} also shows that $H(n)$ can only obtain the values $0$ or $1.$ 

The preceding sections lead us to Theorem~\ref{thm:main2} which demonstrates that the values of $H(n)$ determine the parity of $U(n)$
\begin{equation}\label{eq:H-par}
U(n) \equiv 1 - H(n) \pmod{2}.
\end{equation}

Section ~\ref{sec:zerosandones} shows that $a(n)$ induces a refinement on the partition $I_m,$ into sets, conveniently denoted: $\mathcal{J}$ and $\mathcal{K}$. It is shown that $H(n)$ is $1$ precisely when $n\in\mathcal{K}$ (odd offset) and $0$ when $n\in\mathcal{J}$ (even offset). Theorem \ref{thm:main1} then provides explicit closed-form formulas for the $n$-th elements of the sets $\mathcal{J}$ and $\mathcal{K}.$ Since $H(n)$ determines the parity of $U(n)$, this theorem completely characterizes the sets of all natural numbers for which Redheffer matrices have an odd or even number of unit entries, and thus completes the solution to the Redheffer Matrix Parity Problem.

\section{Notations and General Conventions}\label{sec:notation}
Let $\mathbf{N}=\{0,1,2,\ldots\}$, $\mathbf{Q}$, and $\mathbf{C}$ denote the natural numbers (including 0), the rational numbers, and the complex numbers, respectively. 
Let $x$, $z$, $w$ be algebraically independent indeterminates. Unless stated otherwise, all generating functions are taken as formal power series in $x$ with coefficients in $\mathbf{C}[z,w]$; that is, our ambient ring is $\mathbf{C}[z,w][[x]] \subset \mathbf{C}[[x,z,w]]$.
We also use the polynomial rings $\mathbf{C}[w]$, $\mathbb{C}[z,w]$, and $\mathbf{C}[z,w][x]$.
Via the inclusion $\mathbf{Q}\hookrightarrow \mathbf{C}$ we regard
$$
\mathbf{Q}[z,w]\subset \mathbf{C}[z,w]\subset \mathbf{C}[z,w][x]\subset \mathbf{C}[z,w][[x]]
$$
as subrings. 
Specialization acts on coefficients: for $a,b\in\mathbf{C}$ and $F(x;z,w)\in \mathbf{C}[z,w][[x]]$, we write $F(x;a,b)$ for the image under the $\mathbf{C}$-algebra map $\mathbf{C}[z,w]\to\mathbf{C}$ sending $z\mapsto a$, $w\mapsto b$, giving an element of $\mathbf{C}[[x]]$.  Coefficient extraction is denoted by brackets: $[x^n]H(x)$ is the coefficient of $x^n$.

\section{The Bijective Map \texorpdfstring{$f$}{f} and the Partition \texorpdfstring{$I_m$}{I\_m}}\label{sec:mappartition}
Define the map
$$
f \colon \mathbf{N} \;\longrightarrow\;
\{(m,j)\,\mid\,m\in\mathbf{N},\,0 \le j \le 2m\}
$$
by
$$
f(n) = \bigl(\lfloor \sqrt{n}\rfloor,\;n-\lfloor \sqrt{n}\rfloor^2\bigr).
$$
It is straightforward to check that $f$ is a bijection: \textit{Surjectivity}: For any pair $(m,j)$ with $0 \le j \le 2m$, the integer $n = m^2 + j$ maps to $(m,j)$.
\textit{Injectivity}: If $f(n_1) = f(n_2)$, then $\lfloor \sqrt{n_1}\rfloor = \lfloor \sqrt{n_2}\rfloor$ and $n_1 - m^2 = n_2 - m^2,$ implying $n_1 = n_2$. Hence each $n \in \mathbf{N}$ has a unique ordered pair $(m,j)$ with $0 \le j \le 2m$. Equivalently,
\begin{equation}
n = m^2 + j,
\end{equation}
where $m =\lfloor \sqrt{n}\rfloor,$ with $0 \le j < 2m+1.$
The following example illustrates $f:$ if $n=14$, then $\lfloor \sqrt{14}\rfloor = 3, 14 - 9 = 5,$ and $f(14)=(3,5).$ The first coordinate of $f(n)$ is used to define a partition of $\mathbf{N}$. Specifically, for each $m\in\mathbf{N}$, set
$$
I_{m} = \{n \in \mathbf{N} \mid f(n)=(m,j)\text{ for some } 0\le j \le 2m\}.
$$
Because $f$ is a bijection, these sets $\{I_{m}\}_{m=0}^{\infty}$ are pairwise disjoint and their union is all of $\mathbf{N}$. Thus the family $\{I_m\}_{m=0}^{\infty}$ forms a \emph{partition} of $\mathbf{N}$. We demonstrate the partition with the following example. If $m=1,$ then the block is $I_1 = \{ n \in \mathbb{N} : n = 1^2 + j,\ 0 \leq j \leq 2 \cdot 1 \}$. Thus $j \in \{0,1,2\}$ and the corresponding $n$ are $1^2+0=1$, $1^2+1=2$, and $1^2+2=3$; hence $I_1 = \{1,2,3\}$. For small values of $m$ we have:
\begin{center}
\begin{tabular}{c c c}
\(\displaystyle
\begin{aligned}
I_0 &= \{0\},\\[1mm]
I_1 &= \{1,\,2,\,3\},\\[1mm]
I_2 &= \{4,\,5,\,6,\,7,\,8\},\\[1mm]
I_3 &= \{9,\,10,\,11,\,12,\,13,\,14,\,15\}\\[1mm]
\vdots
\end{aligned}
\)
&
&
\end{tabular}
\end{center}

\section{The Generating Function \texorpdfstring{$F(x;z,w)$}{F(x;z,w)}}\label{sec:ogen-F}
We wish to write both $n$ and the pair $(m,j)$ in one generating function. We begin by introducing algebraically independent indeterminates $x,z$ and $w$ in the formal power–series ring $\mathbf{C}[[x,z,w]],$ letting the exponent of $x$ mark $n,$ the exponent of $z$ track $m,$ and the exponent of $w$ record the offset of $j$ within that block, and set
$$
F(x; z, w)
=
\sum_{n=0}^{\infty}
x^{\,n}\,
z^{\,m}\,
w^{\,j}.
$$
In this sum, each $n$ corresponds uniquely to its pair $(m,j)$ as above.

\begin{lemma}\label{lem:trivariate_F_identity}
In $\mathbf{C}[[x,z,w]]$ one has:
\begin{equation}\label{eq:F-closed-factor}
F(x;z,w) = \frac{1}{1-xw} \sum_{m=0}^{\infty} x^{m^{2}} z^{m} \bigl( 1 - (xw)^{2m+1} \bigr).
\end{equation}
\end{lemma}

\begin{proof}
Every $n\ge0$ can be written uniquely as $n=m^{2}+j$ with $m\ge0$ and $0\le j\le2m$, hence
\begin{equation}\label{eq:F-explicit}
F(x;z,w)=\sum_{m=0}^{\infty}\sum_{j=0}^{2m} x^{m^{2}+j} z^{m} w^{j}.
\end{equation}
For fixed $m$, note that $x^{m^2+j}z^{m}w^{j}=x^{m^2}x^{j}z^{m}w^{j}=x^{m^2}z^{m}(x^{j}w^{j}),$ and so Equation ~\ref{eq:F-explicit} can be rewritten as
\begin{equation}\label{eq:F-geometric-setup}
\sum_{m=0}^{\infty}x^{m^2}z^{m}\left(\sum_{j=0}^{2m}(xw)^{j}\right)
\end{equation} 
The parenthetical term is the geometric sum and is equal to 
\begin{equation*}
\sum_{j=0}^{2m}(xw)^{j}=\frac{1-(xw)^{2m+1}}{1-xw}\qquad\text{in }\mathbf{C}[x,w],
\end{equation*}
Substituting back into ~\ref{eq:F-geometric-setup} yields
\begin{equation}\label{eq:F-geometric}
\sum_{m=0}^{\infty}x^{m^2}z^{m}\left(\frac{1-(xw)^{2m+1}}{1-xw}\right)
\end{equation}
which gives \eqref{eq:F-closed-factor}.
\end{proof}

In $\mathbf{C}[[x,z,w]]$ the series $1-xw$ has constant term $1$ and is therefore a unit, so
$$
\frac{1}{1-xw} = \sum_{t \geq 0} (xw)^t.
$$
In $\mathbf{C}[z,w][[x]]$ we may write, by Lemma~\ref{lem:trivariate_F_identity},
$$
F(x;z,w) = \frac{1}{1-xw} \sum_{m=0}^{\infty}x^{m^2} z^m \left(1 - (xw)^{2m+1}\right).
$$
For each $m$, the factor $1 - (xw)^{2m+1}$ is divisible by $1-xw$ (indeed $1 - u^{2m+1} = (1-u)(1+u+\cdots+u^{2m})$ with $u=xw$). Hence the entire sum in parentheses lies in the principal ideal $(1-xw)$, and the prefactor $1/(1-xw)$ cancels. After cancellation we can specialize $w=1/x$ in the Laurent series ring $\mathbf{C}[z]((x))$, obtaining
$$
F(x;z,x^{-1}) = \sum_{m=0}^{\infty}(2m+1) \, z^m \, x^{m^2}.
$$

Thus the apparent pole at $xw=1$ is removable along the specialization $w=1/x$.

\begin{lemma}\label{lem:conv_F}
Fix $(z,w)\in\mathbf C^{2}$. Then:
\begin{itemize}
\item[(\textit{i})] If $|x|<1$, the series $$\sum_{m=0}^{\infty}x^{m^{2}}z^{m}\bigl(1-(xw)^{2m+1}\bigr)$$ converges absolutely.
\item[(\textit{ii})] If $|x|>1$, $z\neq 0$, and $xw\neq 1$, then the general term does not tend to $0$, hence the series diverges.
\item[(\textit{iii})] On the unit circle $|x|=1$, absolute convergence can hold for special cases.
\end{itemize}
In particular, $|x|<1$ is the largest $x$--region on which absolute convergence holds for \emph{all} parameter choices $(z,w)$.
\end{lemma}

\begin{proof}
Starting from the closed form ~\ref{eq:F-geometric} we distribute inside the summation and compute the second term in the parentheses:
$$
x^{m^{2}} z^{m}(xw)^{2m+1}= x^{m^{2}} z^{m}\,x^{2m+1}w^{2m+1}= x^{m^{2}+2m+1}\,z^{m}\,w^{2m+1}.
$$
If we set
$$
T_m(x,z,w):=x^{m^{2}}z^{m}-x^{m^{2}+2m+1}z^{m}w^{\,2m+1}\qquad(m\ge0),
$$
then
\begin{equation}\label{eq:F-T}
F(x;z,w)=\frac{1}{1-xw}\sum_{m=0}^{\infty}T_m(x,z,w).
\end{equation}
Since $|1-(xw)^{2m+1}|\le 1+|xw|^{2m+1}$, we obtain the crude bound 
\begin{align*}
|T_m(x,z,w)|&\le |x|^{m^{2}}|z|^{m}\bigl(1+|xw|^{2m+1}\bigr)\\
& \le 2\,|x|^{m^{2}}|z|^{m}\max\{1,|w|^{2m+1}\}.
\end{align*}
Fix any $(z,w)\in\mathbf C^{2}$. For \emph{(i)} we see that if $|x|<1$ we have $|x|^{m^{2}}\le e^{m^{2}\log|x|}$ with $\log|x|<0$; hence $|x|^{m^{2}}$ decays super-exponentially. Because the other factors grow at most exponentially in $m$, the general term $T_m(x,z,w)$ tends to $0$ super-exponentially and the series $\sum_{m=0}^{\infty}T_m(x,z,w)$ converges absolutely (keeping in mind the prefactor $(1-xw)^{-1}$). For \emph{(ii)}, assume $|x|>1$, $z\neq0$, and $xw\neq1$; set $c=\log |x|>0$ and $d=\log |z|$. Note that $|x|^{m^{2}}|z|^{m}=e^{cm^{2}+dm}.$ The quadratic term $cm^{2}$ eventually dominates the linear term $dm$, while $1-(xw)^{2m+1}$ is bounded away from $0$ for infinitely many $m$ when $xw\neq1$, the terms $T_m(x,z,w)$ do not tend to $0$; hence the sum diverges. Finally \emph{(iii)} can be seen by observing the following: take $x=-1$ and $z=w=1$ then $T_m(-1,1,1)=2(-1)^{m^2}$, so $|T_m|=2$ for all $m$. Hence the terms do not tend to $0$ and the series cannot converge absolutely on $|x|=1$. All of this shows that absolute convergence occurs precisely when $|x|<1$ is the largest $x-$region on which absolute convergence holds for all $(z,w)$.
\end{proof}

\begin{lemma}\label{lem:degeneratelocus}
For fixed $(z,w)\in\mathbf{C}^2$, if $F(x;z,w)$ is rational in $x$ then $z=0$ or $w=\pm1$ and $z=w.$ Conversely if $z=0$ or $w=\pm1$ and $z=w,$ then $F(x;z,w)$ is rational. In particular, outside these cases the coefficient sequence of $F(x;z,w)$ does not satisfy any nontrivial linear recurrence with constant coefficients.
\end{lemma}
\begin{proof}
It helps to introduce an auxiliary function and work the lemma through it. Starting from Lemma~\ref{lem:conv_F}, Equation ~\ref{eq:F-T} write
$$
F(x;z,w)=\frac{1}{1-xw}\sum_{m=0}^{\infty}T_m(x,z,w).
$$
Split and reindex the second sum, isolate the $m=0$ term and combine the rest:
\begin{align*}
\sum_{m=0}^{\infty}T_m(x,z,w)
&=\sum_{m=0}^{\infty}x^{m^2}z^m\;-\;\sum_{m=0}^{\infty}x^{m^2+2m+1}z^m w^{2m+1}\\
&=\sum_{m=0}^{\infty}x^{m^2}z^m\;-\;\sum_{r=1}^{\infty}x^{r^2}z^{\,r-1} w^{2r-1},\\
&=1+\sum_{m=1}^{\infty}\Bigl(x^{m^2}z^m - x^{m^2}z^{\,m-1}w^{2m-1}\Bigr)\\
&=1+\sum_{m=1}^{\infty}x^{m^2}z^{\,m-1}\bigl(z-w^{2m-1}\bigr).
\end{align*}
Hence
\begin{equation*}
F(x;z,w)=\frac{1}{1-xw}\Bigl(1+\sum_{m=1}^{\infty}x^{m^2}z^{\,m-1}\bigl(z-w^{2m-1}\bigr)\Bigr).
\end{equation*}
Set 
\begin{equation}\label{eq:B}
B(x):=\sum_{m=1}^{\infty} x^{m^2}\,b_m
\end{equation}
with $b_m:=z^{m-1}(z-w^{2m-1}).$
Then 
\begin{equation}\label{eq:F-B}
F(x;z,w)=\frac{1+B(x)}{1-xw}.
\end{equation} 
Suppose $F(x;z,w)$ is rational in $x$. From \eqref{eq:F-B} we have
$B(x)$ is rational. Let $s(n)$ be its coefficients in $\mathbf{C}[[x]]$: $s(m^2)=b_m$ and $s(n)=0$ for non squares. Rationality implies a constant coefficient recurrence $s(n)=\sum_{i=1}^r c_i s(n-i)$ for all large $n$. Choose $m>r$. Since $m^2-r,\ldots,m^2-1$ are not squares, $s$ vanishes there, so the recurrence forces $s(m^2)=0$, i.e.\ $b_m=0$. Thus $b_m=0$ for all sufficiently large $m$. If $z\neq 0$, then for all $m\geq M$ we have $z=w^{2m-1}$. Comparing $m$ and $m+1$ gives $w^{2m-1}=w^{2m+1}$, hence $w^2=1$; therefore $w=\pm 1$ and then $z=w$. Now if $z=0$, then by Equation \eqref{eq:F-B}, we have $b_1=-w$ and $b_m=0$ for $m\ge2$, and by Equation \eqref{eq:B} we have $B(x)=-wx$ so that
$$
F(x;0,w)=\frac{1-wx}{1-wx}=1 \quad \text{in }\mathbf C[[x]].
$$
and is rational. If $w=\pm1$ and $z=w$, then $F(x;w,w)=1/(1-xw)$. These are rational.
\end{proof}

\begin{lemma}\label{lem:F_not_rational}
$F(x;z,w)$ is \emph{not} a rational function of $x$ with coefficients in $\mathbf Q(z,w)$.
\end{lemma}

\begin{proof}
Work over the field $\mathbf{Q}(z,w)$, where $z,w$ are algebraically independent.
From equations ~\ref{eq:B} and ~\ref{eq:F-B} we see that if $F(x;z,w)$ were rational over $\mathbf{Q}(z,w)(x)$ then $B(x)=(1-xw)F(x;z,w)-1$ would also be rational over $\mathbf{Q}(z,w)(x)$. Just the same as Lemma ~\ref{lem:degeneratelocus} write $s(n) =[x^{n}]B(x)$ with $s(m^{2})=b_m$ and $s(n)=0$ for non-squares.
A rational generating function over a field has a coefficient sequence satisfying a linear recurrence with constant coefficients:
there exist $r\ge1$ and $c_1,\dots,c_r\in \mathbf{Q}(z,w)$ such that
$$
s(n)=\sum_{i=1}^{r}c_i\,s(n-i)\qquad\text{for all sufficiently large }n.
$$
Choose $m>r$. Then the integers $m^{2}-r,\dots,m^{2}-1$ are not squares, hence $s(m^{2}-i)=0$ for $1\le i\le r$.
Applying the recurrence at $n=m^{2}$ yields $s(m^{2})=0$ in $\mathbf{Q}(z,w)$.
But $s(m^{2})=b_m=z^{m-1}(z-w^{2m-1})\in\mathbf{Q}(z,w)$ is \emph{not} the zero element of $\mathbf{Q}(z,w)$ -- it is a nonzero polynomial in the algebraically independent variables $z,w$. This contradiction shows $B(x)$ is not rational over $\mathbf{Q}(z,w)(x)$, and therefore $F(x;z,w)$ is not rational over it as well.
\end{proof}

\begin{remark}\label{rem:F_transcendental}
We claim $F(x;z,w)$ is \emph{transcendental} over $\mathbf{Q}(z,w)(x).$ That is it satisfies no nonzero polynomial equation
$P(x,z,w,Y)=0$ with $P\in\mathbf{Q}(z,w)[x,Y]$.
Assume, towards a contradiction, that $F$ is algebraic over $\mathbf{Q}(z,w)(x)$.
Then there exists a nonzero $Q(x,z,w,Y)\in\mathbf{Q}(z,w)[x,Y]$ with 
$Q\bigl(x,z,w,F(x;z,w)\bigr)=0.$ Clear denominators by choosing a nonzero $R(z,w)\in\mathbf{Q}[z,w]$ so that
$$
P(x,z,w,Y):=R(z,w)\,Q(x,z,w,Y)\in\mathbf{Q}[z,w][x,Y],
$$
and $P\bigl(x,z,w,F(x;z,w)\bigr)=0$.
Because $\mathbf{Q}[z,w]$ is a UFD, let $k\ge0$ be the largest integer such that
$w^{k}$ divides \emph{all} coefficients of $P$ (viewed in $\mathbf{Q}[z,w]$); replace $P$ by $P/w^{k}$.
Then the specialization $P(x,z,0,Y)$ is a \emph{nonzero} polynomial in $\mathbf{Q}[z][x,Y]$, and
$
P\bigl(x,z,0,F(x;z,0)\bigr)=0.
$
Now regard the coefficients of $P(x,z,0,Y)$ as polynomials in $z$; let $\ell\ge0$ be the largest integer such that $(z-1)^{\ell}$ divides all those coefficients in $\mathbf{Q}[z]$, and replace $P(x,z,0,Y)$ by
$P(x,z,0,Y)/(z-1)^{\ell}$.
Then $P(x,1,0,Y)\in\mathbf{Q}[x,Y]$ is \emph{nonzero}, and we obtain a nontrivial algebraic relation
$$
P\bigl(x,1,0,F(x;1,0)\bigr)=0.
$$
But $F(x;1,0)=\sum_{m\ge0}x^{m^{2}}=:S(x)$ has coefficients in the finite set $\{0,1\}$, hence by Fatou’s theorem (1906) \footnote{See, e.g., G.~Pólya and G.~Szegő, \emph{Aufgaben und Lehrsätze aus der Analysis}, vol.~II, Aufgabe 236. Note Fatou’s 1906 theorem disallows finitely valued algebraic series that aren’t rational}, a finitely valued algebraic power series is rational. By Lemma~\ref{lem:degeneratelocus}, $S(x)$ is not rational; thus $S(x)$ cannot be algebraic--- a contradiction. Therefore $F(x;z,w)$ is transcendental over $\mathbf{Q}(z,w)(x)$.
\end{remark}

\begin{remark}
A reader might object to the worked-out derivation in the preceding discussion. In addition to Remark ~\ref{rem:scope}, I might add the overly pedantic work serves as part of my own self-education. Indeed the degeneracy locus of $F$ made analytic arguments more delicate, and harder to work through than I had anticipated. At a minimum the detailed derivation serves as a warmup for the calculations that are to follow. Finally the astute reader might have observed the appearance of the expression $\sum_{m=0}^{\infty}x^{m^2}$, which suggests the use of the modular function theory. For now, I prefer to work in the classical generating function setting. The Jacobi Triple Product can wait.
\end{remark}

\section{Linear Projection by \texorpdfstring{$\varphi$}{phi} and The Series \texorpdfstring{$G(x)$}{G(x)}}\label{sec:linprojection}
\begin{lemma}\label{lem:parity-projection} 
The following holds:
\begin{enumerate}[label=(\roman*)]
\item \emph{Existence/uniqueness:} There exists a unique $\mathbf{C}$–linear functional $L_w:\mathbf{C}[w]\to\mathbf{C}$ such that $L_w(1)=0,$ and for $r \geq 1,$ $L_w(w^{2r+1})=+1,$ and $L_w(w^{2r})=-1.$ 
\item \emph{Specialization identity:} Define $\varphi:\mathbf{C}[[x,z,w]]\to\mathbf{C}[[x]]$ on monomials by
$ \varphi \bigl(x^n z^m w^j\bigr)=x^n\,L_w(w^j),$ and extend to all series by linearity. Then
$$
\varphi\bigl(F(x;z,w)\bigr)=F(x;1,0)-F(x;1,-1),
$$
\item \emph{Coefficient extraction:} If $$G(x):=\varphi \bigl(F(x;z,w)\bigr)=\sum_{n=0}^\infty a(n)\,x^n$$, then the coefficients satisfy
$$
a(n)=
\begin{cases}
0,& \text{if }n\text{ is a perfect square},\\[3pt]
(-1)^{\,j-1},& \text{if }n=m^2+j\text{ with }1\le j\le 2m.
\end{cases}
$$
\end{enumerate}
\end{lemma}

\begin{proof}
\emph{Existence/uniqueness of $L_w$.}
Since $\{1,w,w^2,\dots\}$ is a standard $\mathbf{C}$–basis of $\mathbf{C}[w]$, the listed values determine at most one linear functional.
Define $L_w(P):=P(0)-P(-1)$. Then $L_w(1)=1-1=0$, $L_w(w)=0-(-1)=1$, $L_w(w^2)=0-1=-1$, and in general
\begin{equation}\label{eq:L-1}
L_w(w^j)=0-(-1)^j=(-1)^{j-1}
\end{equation}
for $j\ge1$, so $L_w$ has the required values. \emph{Specialization identity.}
Reindex $F$ by $n=m^2+j$ with $m\ge0$ and $0\le j\le 2m$:
\begin{align*}
F(x;z,w)&=\sum_{m=0}^\infty\sum_{j=0}^{2m}x^{\,m^2+j}\,z^{\,m}\,w^{\,j}\\[1 ex]
&=\sum_{m=0}^\infty x^{\,m^2}z^{\,m}\Bigl(\sum_{j=0}^{2m} (xw)^{j}\Bigr).
\end{align*}
For each $m$ the inner sum is a \emph{finite} polynomial in $w$, hence
\begin{equation}\label{eq:F-phi-1}
\varphi\bigl(F(x;z,w)\bigr)
=\sum_{m=0}^\infty x^{\,m^2}\,L_w\Bigl(\sum_{j=0}^{2m} (xw)^{j}\Bigr).
\end{equation}
Using $L_w(P)=P(0)-P(-1)$ gives
$$
L_w\Bigl(\sum_{j=0}^{2m} (xw)^{j}\Bigr)
=\sum_{j=0}^{2m} (x\cdot 0)^{j}-\sum_{j=0}^{2m} (x\cdot(-1))^{j},
$$
hence
\begin{align}\label{eq:F-phi-2}
\varphi\bigl(F(x;z,w)\bigr)&=\sum_{m=0}^\infty x^{\,m^2}-\sum_{m=0}^\infty x^{\,m^2}\!\left(\sum_{j=0}^{2m}(-x)^{j}\right)\\[1ex]
&=F(x;1,0)-F(x;1,-1),\label{eq:phi-F}
\end{align}
which lies in $\mathbf{C}[[x]]$.
\emph{Coefficient extraction.}
Recall equation ~\ref{eq:F-phi-1} and write 
\begin{equation}\label{eq:G-1}
G(x)=\varphi\bigl(F(x;z,w)\bigr)=\sum_{m=0}^\infty\sum_{j=0}^{2m} L_w(w^{\,j})\,x^{\,m^2+j}.
\end{equation}
Fix $n\ge0$ and write $n$ uniquely as $n=m^2+j$ with $m\ge0$ and $0\le j\le 2m$.
Then $x^n$ appears exactly once in the double sum, so
$$
a(n)=[x^n]\,G(x)=L_w(w^{\,j}).
$$
Using the values of $L_w$, see ~\ref{eq:L-1}, already established yields the stated cases.
\end{proof}

\begin{remark} Let $\odot$ denote the Hadamard product in the variable $w$:
if $A(w)=\sum_{j\ge0}a_j w^j$ and $B(w)=\sum_{j\ge0}b_j w^j$, then
$(A\odot B)(w)=\sum_{j\ge0}a_j b_j\,w^j$.
Set
$$
\Lambda(w):=\sum_{j\ge0}L_w(w^j)\,w^j=\sum_{j\ge1}(-1)^{j-1}w^j=\frac{w}{1+w}.
$$
Then for every polynomial $P(w)\in\mathbf C[w]$ one has
$L_w(P)=(P\odot \Lambda)(1),$ since $P\odot\Lambda$ is again a polynomial 
and evaluation at $w=1$ is formal. Equivalently,
$L_w(P)=\operatorname{CT}_w\!\bigl   (P(w^{-1})\,\Lambda(w)\bigr).$
\end{remark}

\begin{corollary}\label{cor:a-nonrecurrence}
The sequence $a(n)$ does not satisfy any nontrivial linear recurrence with constant coefficients.
\end{corollary}

\begin{proof}
Using Lemma~\ref{lem:trivariate_F_identity} with $z=1$ and the identity from Equation~\ref{eq:phi-F} we compute
\begin{align*}
(1+x)G(x)&=1+(x-1)\sum_{m=0}^{\infty}x^{m^2}.\\
\end{align*}
By Lemma~\ref{lem:degeneratelocus}, the series $F(x;1,0)=\sum_{m=0}^{\infty}x^{m^2}$ is not rational; hence the right-hand side above is not rational. Therefore $(1+x)G(x)$ is not rational, and consequently $G(x)$ itself is not rational. It is standard that an ordinary generating function is rational if and only if its coefficient sequence satisfies a linear recurrence with constant coefficients. Thus $a(n)$ is not linearly recurrent.
\end{proof}

\begin{lemma}
\begin{equation}\label{eq:ogen}
G(x) = \sum_{m=1}^{\infty} \frac{x^{m^2+1} \, (1 - x^{2m})}{1+x}.
\end{equation}
\end{lemma}
\begin{proof}
Set $z=1$ in $F$ and use Lemma \ref{lem:trivariate_F_identity}:
$$
F(x;1,w)=\frac{1}{1-xw}
\sum_{m=0}^{\infty}x^{m^{2}}\bigl(1-(xw)^{2m+1}\bigr).
$$
Linear projection gives
$G(x)=F(x;1,0)-F(x;1,-1)$.
Substituting $w=0$ and $w=-1$ gives
$$
G(x)
=\sum_{m=0}^{\infty}x^{m^{2}}
-\frac{1}{1+x}\sum_{m=0}^{\infty}x^{m^{2}}\bigl(1+x^{2m+1}\bigr).
$$
The $m=0$ terms cancel, so starting at $m=1$ and simplifying yields,
$$
G(x)=\sum_{m=1}^{\infty}
\frac{x^{m^{2}+1}\,\bigl(1-x^{2m}\bigr)}{1+x},
$$
which is exactly \eqref{eq:ogen}.
\end{proof}

\section{The Series \texorpdfstring{$H(x)$}{H(x)} and The Partial Sums \texorpdfstring{$H(n)$}{H(n)}}\label{sec:ogen-H}
The generating function for the partial sums is given by
$$
H(x) = \frac{G(x)}{1-x} = \frac{1}{1-x}\sum_{n=0}^\infty a(n)x^n.
$$

\begin{theorem}\label{thm:parsum}
The following holds:
\begin{enumerate}[label=(\roman*)]
\item
$$
H(x)=\sum_{n=0}^\infty\Bigl(\sum_{k=0}^{n}a(k)\Bigr)x^n,
$$
so $H(n):=\sum_{k=0}^{n}a(k)$ is the coefficient of $x^n$ in $H(x)$.
\item 
$$
H(n)\;=\;\frac{1-(-1)^{\,n-\lfloor\sqrt{n}\rfloor^{2}}}{2}.
$$
\end{enumerate}
\end{theorem}

\begin{proof}
Since $\dfrac{1}{1-x}=\sum_{r=0}^\infty x^r$ in $\mathbf{C}[[x]]$, the Cauchy product gives
$$
H(x)=\frac{G(x)}{1-x}
=\Bigl(\sum_{n=0}^\infty a(n)\,x^n\Bigr)\Bigl(\sum_{r=0}^\infty x^r\Bigr)
=\sum_{n=0}^\infty\Bigl(\sum_{k=0}^{n}a(k)\Bigr)x^n,
$$
hence $H(n)=\sum_{k=0}^{n}a(k)$. We recall that $n=m^{2}+j$ with $m=\lfloor\sqrt{n}\rfloor$ and $0\le j\le 2m$. Inside the block
$I_m=\{m^{2},\dots,(m+1)^{2}-1\}$ we have $a(m^{2})=0$ and, for $1\le t\le 2m$,
$a(m^{2}+t)=(-1)^{t-1}$. Therefore the within-block partial sums satisfy
$$
\sum_{t=0}^{j} a(m^{2}+t)=\sum_{t=1}^{j}(-1)^{t-1}
=
\begin{cases}
1,& j\ \text{odd},\\
0,& j\ \text{even}.
\end{cases}
$$
Each full block sums to $0$, so the global partial sum up to $n$ equals this within-block sum. Hence $H(n)=1$ if $j$ is odd, and $0$ if $j$ is even. Equivalently
$$
H(n)=\dfrac{1-(-1)^{\,j}}{2}
=\dfrac{1-(-1)^{\,n-\lfloor\sqrt{n}\rfloor^{2}}}{2}.
$$
\end{proof}


\begin{corollary}\label{cor:H-nonrecurrence}
The values $H(n)$ do not satisfy any nontrivial linear recurrence with constant coefficients.
\end{corollary}

\begin{proof}
Since $H(x)=G(x)/(1-x)$ and $1/(1-x)$ is rational, the non-rationality of $G(x)$ implies the non-rationality of $H(x)$, hence $H(n)$ is not linearly recurrent.
\end{proof}

It helps to write $a(n)$ and $H(n)$ out for small values of $n$ to get a better handle on them:
\vspace{1em} 
\begin{center}
\footnotesize
\renewcommand{\arraystretch}{1.2} 
\setlength{\tabcolsep}{5pt} 
\begin{tabular}{c|*{13}{c}}
$n$    & 0 & 1 & 2 & 3 & 4 & 5 & 6 & 7 & 8 & 9 & 10 & 11 & $\cdots$ \\ \hline
$a(n)$ & 0 & 0 & 1 & $-1$ & 0 & 1 & $-1$ & 1 & $-1$ & 0 & 1 & $-1$ & $\cdots$ \\ \hline
$H(n)$ & 0 & 0 & 1 & 0 & 0 & 1 & 0 & 1 & 0 & 0 & 1 & 0 & $\cdots$
\end{tabular}
\end{center}
\vspace{1em} 

\section{The Congruence Equation for \texorpdfstring{$U(n)$ and $H(n)$}{U(n) and H(n)}}\label{sec:par-U-n}
\begin{theorem}\label{thm:main2}
The following holds:
$$
U(n) \equiv 1 - H(n) \pmod{2}.
$$
\end{theorem}
\begin{proof}
By the Offset Rule from the Introduction, $U(n)$ is odd if and only if $j$ is even, equivalently 
$U(n)\equiv j+1\pmod{2}$. 
By Theorem~\ref{thm:parsum}(ii), $H(n)=1$ if and only if $j$ is odd, so $1-H(n)\equiv j+1\pmod{2}$. 
Hence $U(n)\equiv 1-H(n)\pmod{2}$.
\end{proof}

\section{Zeros and Ones}\label{sec:zerosandones}
Following Theorem~\ref{thm:main2}, it is desirable to know the $n$ for which $H(n)=0$ or $H(n)=1.$ Observe that inside each block $I_m$ the values of $a(n)$ split $I_m$ into two disjoint subsets:
\begin{equation}\label{eq:J-def}
J_m=\{\,n\in I_m : a(n)\le0\}=\bigl\{\,m^2 + j : 0\le j\le 2m,\;j\text{ even}\bigr\}
\end{equation}
and
\begin{equation}\label{eq:K-def}
K_m=\{\,n\in I_m : a(n)=+1\}=\bigl\{\,m^2 + j : 1\le j\le 2m,\;j\text{ odd}\bigr\}
\end{equation}
Clearly
$$
I_m=J_m \cup K_m.
$$
Notice that $K_{0}$ is empty but for $m\ge 1,$ $J_{m}$ and $K_{m}$ are nonempty. Thus, if one omits $K_{0}$, the family
$$
J_{0} \cup \bigcup_{m=1}^{\infty} \{\,J_{m},\,K_{m}\}
$$
is a proper \emph{refinement} of $I_{m}$ in the standard sense (every block is nonempty, they are pairwise disjoint, and their union is $\mathbf{N}$). It will be convenient to write 
$$
\mathcal{J}=\bigcup_{m=0}^{\infty} J_m, \quad \text{and} \quad \mathcal{K}=\bigcup_{m=1}^{\infty} K_m.
$$
We illustrate the preceding observation by writing out $J_m,$ and $K_m:$
\vspace{1em} 
\[
\begin{aligned}
J_0&=\{0\},      &\quad K_0&=\varnothing,\\
J_1&=\{1,3\},    &\quad K_1&=\{2\},\\
J_2&=\{4,6,8\},  &\quad K_2&=\{5,7\},\\
J_3&=\{9,11,13,15\},&\quad K_3&=\{10,12,14\},\\
&\vdots          &&\vdots
\end{aligned}
\]
\vspace{1em} 
It follows that,
\begin{equation}\label{eq:Jset}
\mathcal{J}=\{0,1,3,4,6,8,9,11,13,15,\ldots\},
\end{equation}
and 
\begin{equation}\label{eq:Kset}
\mathcal{K}=\{2,5,7,10,12,14,\ldots\}.
\end{equation}

\noindent As in the analysis of $G(x)$, the following section is worked entirely in the formal power--series ring $\mathbf{C}[[x]]$; all manipulations are purely formal, and inverses are taken only for series with unit constant term. We begin with the generating function for $\mathcal{J}$, which records a term $x^n$ exactly when $n\in\mathcal{J}$. First recall the definition of $\mathcal{J}$ as the disjoint union of sets from ~\ref{eq:J-def} and note this shows its generating function $J(x)$ is
$$
J(x)=\sum_{m=0}^{\infty}\;\sum_{n \in J_m}x^{n}.
$$
We introduce the variable $r=j/2$ to reindex the even offset $j\in\{0,2,\ldots,2m\}$ by $r\in\{0,1,\ldots,m\}$ from which the following is established:
\begin{align*}
J(x)&=\sum_{m=0}^{\infty}\;\sum_{r=0}^{m}x^{\,m^2+2r}\\[1ex]
&=\sum_{m=0}^{\infty}x^{m^2}\sum_{r=0}^{m}x^{2r}\\[1ex]
&=\frac{1}{1-x^2}\sum_{m=0}^{\infty}x^{m^2}\Bigl(1-x^{2(m+1)}\Bigr).
\end{align*}

Since $\mathcal{J}$ and $\mathcal{K}$ form a partition of $\mathbf{N}$, we have
$$
J(x)+K(x)=\sum_{n=0}^{\infty}x^n=\frac{1}{1-x},
$$
hence
$$
K(x)=\frac{1}{1-x}-J(x).
$$
Grouping $\mathcal{K}$ by blocks gives
\begin{align*}
K(x)&=\sum_{m=1}^{\infty}\;\sum_{r=0}^{m-1}x^{\,m^2+(2r+1)}\\[1ex]
&=\sum_{m=1}^{\infty}x^{m^2+1}\sum_{r=0}^{m-1}x^{2r}\\[1ex]
&=\frac{1}{1-x^2}\sum_{m=1}^{\infty}x^{m^2+1}\Bigl(1-x^{2m}\Bigr).
\end{align*}
Using $1/(1-x^2)=1/[(1-x)(1+x)]$ and the already-established formula  from ~\ref{eq:ogen}
we obtain
$$
K(x)=\frac{1}{(1-x)(1+x)}\sum_{m=1}^{\infty}x^{m^2+1}\bigl(1-x^{2m}\bigr)
=\frac{G(x)}{1-x}.
$$
Since $H(x)=G(x)/(1-x)$ is the generating function of the partial sums $H(n)$, we conclude
\begin{equation}
K(x)=H(x).
\end{equation}\label{eq:K}
In particular
\begin{equation}\label{eq:HJK}
H(n)=[x^{n}]\,H(x)=[x^{n}]\,K(x)=
\begin{cases}
1,& \text{if } n\in\mathcal{K}\ \text{(odd offset)},\\[2pt]
0,& \text{if } n\in\mathcal{J} \ \text{(even offset)},
\end{cases}
\end{equation}
\begin{remark}
The specializations of $F$ make explicit the refinement $I_m$ by the union of the disjoint sets $J_m$ and $K_m$. Setting $w=1$ suppresses the offset and lists every $n$, so $F(x;1,1)=1/(1-x)$. The even offset and odd offset parts arise from the two linear projections $1/2\bigl(F(x;1,1)+F(x;1,-1)\bigr)=J(x)$ and $1/2\bigl(F(x;1,1)-F(x;1,-1)\bigr)=K(x)$, respectively, while $F(x;1,0)=\sum_{m=0}^{\infty}x^{m^2}$ keeps only the squares, a proper subset of $\mathcal{J}$. Finally, with $G(x)=F(x;1,0)-F(x;1,-1)$ and $H(x)=G(x)/(1-x)$ one has $H(x)=K(x)$, so $H(n)=1$ exactly for $n\in\mathcal{K}$ and $H(n)=0$ exactly for $n\in\mathcal{J}$.
\end{remark}

\section{\texorpdfstring{$U(n)$}{U(n)} and The Sets \texorpdfstring{$\mathcal{J}$ and $\mathcal{K}$}{J and K}}\label{sec:UrelateJK}

\begin{lemma}\label{lem:npar-noemum}
For $n\ge 1$, $U(n)$ is odd if and only if $n\in\mathcal{J}$ and even if and only if $n\in\mathcal{K}$.
\end{lemma}

\begin{proof}
Suppose $U(n)$ is odd. By Theorem~\ref{thm:main2}, 
$U(n)\equiv 1-H(n)\pmod{2}$, so $1-H(n)\equiv 1\pmod{2}$ and hence $H(n)\equiv 0\pmod{2}$.
By Theorem~\ref{thm:parsum}, $H(n)\in\{0,1\}$, so $H(n)=0$. By \eqref{eq:HJK},
$H(n)=0$ exactly on $\mathcal{J}$; therefore $n\in\mathcal{J}$. Similarly, if $U(n)$ is even, 
then $1-H(n)\equiv 0\pmod{2}$, whence $H(n)\equiv 1\pmod{2}$.
Again $H(n)\in\{0,1\}$, so $H(n)=1$, and by this is equivalent to $n\in\mathcal{K}$. Now, if $n\in \mathcal{J}$ then $H(n)=0$. Following Theorem~\ref{thm:main2}, $U(n)\equiv 1-0\equiv 1 \pmod{2}.$ If 
$n\in \mathcal{K}$ then $H(n)=1$ and $U(n)\equiv 1-1\equiv 0 \pmod{2}.$
\end{proof}

Our immediate task then is to determine the exact form of the elements of $\mathcal{J}$ and $\mathcal{K}$ which will inform us as to which numbers $n$ lead $H(n)$ to evaluate to a zero or a one. 

\section{Enumerating the Elements of \texorpdfstring{$\mathcal{J}$ and $\mathcal{K}$}{J and K}}\label{sec:enumJK}

\begin{theorem}\label{thm:main1}
\itshape Let $j(n)$ and $k(n)$ denote the $(n+1)-$th smallest element of the sets $\mathcal{J},$ and $\mathcal{K}$ respectively. Then the following hold:
\begin{enumerate}
\item[(a)]
$$
j(n) = 2n - \left\lfloor \frac{\sqrt{8n + 1} - 1}{2} \right\rfloor, \quad \text{and} \quad H(j(n))=0.
$$
\item[(b)]
$$
k(n) = 2n + 1 + \left\lceil \frac{\sqrt{8n+9} - 1}{2} \right\rceil, \quad \text{and} \quad H(k(n))=1
$$
\end{enumerate}
\end{theorem}

\begin{lemma}\label{lem:j(n)}
\itshape
Let $j(n)$ denote the $(n+1)-$th smallest element of the set $\mathcal{J}.$ Then:
\begin{equation}
j(n) = 2n - \left\lfloor \frac{\sqrt{8n + 1} - 1}{2} \right\rfloor.
\end{equation}
\end{lemma}
\begin{proof}
Each $J_m$ has exactly $(m+1)$ elements. For a real number $x \ge 0$, define ${J}_{\le x}=\{y \in \mathcal{J} : y \le x \}. $ If we set $m=\lfloor\sqrt{x}\rfloor$ then for each $k=0,1,\dots,m-1$, the entire block $J_k = \{\,k^2, k^2+2,\dots,k^2+2k\}$ lies in $[0,x]$. Since $|J_k|=k+1$,
$$
\sum_{k=0}^{m-1} (k+1) = \frac{m(m+1)}{2}.
$$
Now consider the block $J_m$. The elements of $J_m$ that lie below $x$ must satisfy $m^2 + 2r \;\le\; x,$ which implies $r \le (x - m^2)/2$. Since $r \ge 0$ is an integer, the number of such $r$ is 
$$
\left\lfloor \frac{x - m^2}{2} \right\rfloor + 1 \quad.
$$
Putting these counts together, we get
$$
|\mathcal{J}_{\le x}| = \frac{m(m+1)}{2} + \Bigl\lfloor \tfrac{x - m^2}{2}\Bigr\rfloor +1.
$$
Now, $j(n)$ is the unique $x$ for which $|\mathcal{J}_{\le x}| = n + 1.$ From the counting formula,
$$
\frac{m(m+1)}{2} + \left\lfloor \tfrac{x - m^2}{2} \right\rfloor +1 = n+1,
$$
Subtract 1 from both sides:
$$
\left\lfloor \tfrac{x - m^2}{2} \right\rfloor = n -\frac{m(m+1)}{2}.
$$
Call the right-hand side $r$. Since $0 \le r \le m$, we have that, $(x - m^2)/2$ falls in $[r,\,r+1);$ which implies $x = m^2 + 2r.$ But $r = n - m(m+1)/2$. Hence
$$
j(n) = m^2 + 2\Bigl(n - \tfrac{m(m+1)}{2}\Bigr)=2n-m.
$$
The integer $m$ must satisfy
$$
\frac{m(m+1)}{2} \le n \le \frac{m(m+1)}{2} + m.
$$
Equivalently, $m$ is the largest integer so that $m(m+1)/2\le n$. Since $m(m+1) \le 2n < (m+1)(m+2),$ we see that $m$ is the largest integer solving the inequality $m(m+1) \le 2n.$ Consider the quadratic equation $m^2 + m - 2n = 0.$ By the standard quadratic formula, the solutions are
$$
m = \frac{-1 \pm \sqrt{1 - 4t}}{2}.
$$
In our case, $t = -2n$, so $1 - 4t = 1 - 4(-2n) = 1 + 8n.$ Hence the two real solutions to $m^2 + m - 2n = 0$ are
$$
m = \frac{-1 \pm \sqrt{1 + 8n}}{2}.
$$
Discard the negative root and note that $m$ must lie \emph{at or below} the positive solution. Thus,
$$
m \le \frac{-1 + \sqrt{1 + 8n}}{2}.
$$
The largest integer $m$ satisfying $m(m+1)\le 2n$ is therefore
$$
m =\left\lfloor \frac{\sqrt{8n + 1} - 1}{2}\right\rfloor.
$$
Substitution gives 
\begin{equation}\label{eq:j(k)}
j(n)=2n-\left\lfloor\frac{\sqrt{8n+1}-1}{2}\right\rfloor.
\end{equation}
\end{proof}
The values of $j(n)$ are:
$$
0,1,3,4,6,8,9,11,13,15,16,\ldots,.
$$

\begin{lemma}\label{lem:k(n)}
\itshape
Let $k(n)$ denote the $(n+1)-$th smallest element of the set $\mathcal{K}.$ Then
\begin{equation}
k(n) = 2n + 1 + \left\lceil \frac{\sqrt{8n+9} - 1}{2} \right\rceil,
\end{equation}
\end{lemma}
\begin{proof}
Each block $K_m$ is defined by $K_m=\{\,m^2+1,\; m^2+3,\; \dots,\; m^2+(2m-1)\},$so that $|K_m|=m$. For a real number $x\ge 0$, define $\mathcal{K}_{\le x}=\{\,y\in\mathcal{K}: y\le x\}.$ If we set $m=\lfloor \sqrt{x}\rfloor$, then for each $k=1,2,\dots, m-1$ the entire block $K_k$ lies within $[0,x]$. Since $|K_k|=k$, the contribution of these blocks is
$$
\sum_{k=1}^{m-1} k = \frac{(m-1)m}{2}.
$$
Now consider the block $K_m$. An element of $K_m$ has the form $x=m^2+(2r-1)$ for some integer $r$ with $1\le r\le m$. Such an element satisfies $m^2+(2r-1)\le x$ so
$r\le \frac{x-m^2+1}{2}.$ Thus, the number of elements of $K_m$ that do not exceed $x$ is exactly
$$
\left\lfloor \frac{x-m^2+1}{2} \right\rfloor,
$$
provided that $\left\lfloor \frac{x-m^2+1}{2} \right\rfloor \le m$ (which is automatic when $x\in K_m$). Therefore, the total number of elements of $\mathcal{K}$ that are at most $x$ is
$$
|\mathcal{K}_{\le x}|=\frac{(m-1)m}{2}+\left\lfloor \frac{x-m^2+1}{2} \right\rfloor.
$$
Now, by definition, $k(n)$ is the unique $x\in\mathcal{K}$ for which
$$
|\mathcal{K}_{\le x}| = n+1.
$$
Since $x\in K_m$, we write $x=m^2+(2r-1)$ for some $r$, $1\le r\le m$, so that
$$
\left\lfloor \frac{x-m^2+1}{2} \right\rfloor = r.
$$
Thus, the counting formula becomes
$$
\frac{(m-1)m}{2}+r = n+1.
$$
Solving for $r$, we obtain
$$
r=n+1-\frac{(m-1)m}{2}.
$$
Since $x=m^2+(2r-1)$, it follows that
$$
k(n)=m^2+2\Bigl(n+1-\frac{(m-1)m}{2}\Bigr)-1.
$$
A short calculation shows:
$$
k(n)=m^2+2n+2-(m^2-m)-1=2n+m+1.
$$
The integer $m$ is determined by the requirement that the block $K_m$ contains the $(n+1)$th element. Since the total number of elements in $\mathcal{K}$ up through the entire block $K_m$ is
$$
\frac{(m-1)m}{2}+ m = \frac{m^2+m}{2},
$$
the condition is
$$
\frac{(m-1)m}{2} < n+1\le \frac{m^2+m}{2}.
$$
This inequality is equivalent to
$$
m(m+1)\ge 2(n+1) > m(m-1).
$$
Consider the quadratic equation
$$
m^2+m-2(n+1)=0.
$$
By the standard quadratic formula, the solutions are 
$$
m = \frac{-1 \pm \sqrt{1 - 4t}}{2}.
$$
In this case, $t = -2(n+1)=-2n-1$, so $1 - 4t = 1 - 4(-2(n+1)) = 1 + 8(n+1)=8n+9.$ Its positive solution is
$$
m=\frac{-1+\sqrt{1+8(n+1)}}{2} = \frac{-1+\sqrt{8n+9}}{2}.
$$
Thus, $m$ is the smallest integer greater than or equal to $\frac{-1+\sqrt{8n+9}}{2}$, which can be written as
$$
m=\left\lceil \frac{\sqrt{8n+9}-1}{2} \right\rceil.
$$
Substituting this expression for $m$ into $k(n)=2n+m+1$, we obtain
\begin{equation}\label{eq:k(n)}
k(n)=2n+1+\left\lceil \frac{\sqrt{8n+9}-1}{2} \right\rceil.
\end{equation}
\end{proof}
The first few values are:
$$
2,5,7,10,12,14,17,19,21\ldots.
$$

\subsection{Proof of Theorem ~\ref{thm:main1}}\label{sec:main1}
\begin{proof}
Inside each block $I_m=\{m^2,\dots,(m+1)^2-1\}$ we have $a(m^2)=0$ and 
$a(m^2+j)=(-1)^{j-1}$ for $1\le j\le 2m$, so by Theorem~\ref{thm:parsum}\,(ii) the partial sum
$H(n)$ equals $0$ on even offsets and $1$ on odd offsets. 
By construction $j(n)\in\mathcal{J}$ (even offsets) and $k(n)\in\mathcal{K}$ (odd offsets), hence 
$H(j(n))=0$ and $H(k(n))=1$. 
The explicit formulas for $j(n)$ and $k(n)$ are given by Lemmas~\ref{lem:j(n)} and~\ref{lem:k(n)}.
\end{proof}

\begin{corollary}\label{cor:npar-enum}
For $n\ge 1$, $U\bigl(j(n)\bigr)\equiv 1\pmod{2}$.
For $n\ge 0$, $U\bigl(k(n)\bigr)\equiv 0\pmod{2}$.
\end{corollary}

\begin{proof}
Since $j(n)\in\mathcal J$ for all $n$, \eqref{eq:HJK} gives $H\bigl(j(n)\bigr)=0$.
Then Theorem~\ref{thm:main2} yields
$U\bigl(j(n)\bigr)\equiv 1-H\bigl(j(n)\bigr)\equiv 1\pmod{2}$.
Similarly, $k(n)\in\mathcal K$ for all $n$, so $H\bigl(k(n)\bigr)=1$ and
$U\bigl(k(n)\bigr)\equiv 1-1\equiv 0\pmod{2}$.
\end{proof}

\section{Final Note}\label{sec:finalnote}
We are curious to ask whether the sequence of integers comprising the sets $\mathcal{J}$ and $\mathcal{K}$ are known and previously examined. Answering that question leads directly to a query proposed by Ian G. Connell appearing in the \textit{American Mathematical Monthly} in 1959, \cite{E1382-problem}, concerning a particular integer sequence. A solution to Connell's query was given in \cite{E1382-solution}, the following year. Our own work connects the parity of $U(n)$ to Connell's sequence.

Ian Connell challenged readers to find a closed form for an unusual sequence, $c(n)$. The sequence, which now bears his name, is constructed as follows: begin by writing the first odd number, $1$; then write the next two even numbers after $1$, namely $2,\,4$; then write the next three odd numbers after $4$, namely $5,\,7,\,9$; and so on. The first few terms of the sequence are
$$
1,\, 2,\, 4,\, 5,\, 7,\, 9,\, 10,\, 12,\, 14,\, 16,\, 17,\, \ldots
$$
Andrew Korsak (University of Toronto) provided a solution:
\begin{equation}\label{eq:korsak}
c(n)\;=\;2n-\left\lfloor\frac{1+\sqrt{\,8n-7\,}}{2}\right\rfloor.
\end{equation}
This sequence is listed in the OEIS as \href{https://oeis.org/A001614}{A001614} (\nolinkurl{https://oeis.org/A001614}). 

N.S. Mendelsohn considered a modification of the Connell sequence by replacing the first term with the even number $0$, then writing the next two odd numbers after $0$, namely $1,\,3$, followed by the next three even numbers after $3$, namely $4,\,6,\,8$, and so on. The first few terms of this sequence are
$$ 
0,\, 1,\, 3,\, 4,\, 6,\, 8,\, 9,\, 11,\, 13,\, 15,\, \ldots 
$$
N.S. Mendelsohn provided a solution to his own sequence
\begin{equation}\label{eq:mendelsohn}
m(n)=2n-\left\lfloor \frac{\sqrt{8n+1}-1}{2}\right\rfloor.
\end{equation}
This sequence is listed in the OEIS as \href{https://oeis.org/A133280}{A133280} (\nolinkurl{https://oeis.org/A133280}). 

There is another sequence related to Mendelsohn's sequence namely its complement:
$$
2,\,5,\,7,\,10,\,12,\,14,\,17,\,19,\,21,\,23,\ldots,
$$
which is listed in the OEIS as \href{https://oeis.org/A195437}{A195437} (\nolinkurl{https://oeis.org/A195437}).
In the language developed above, this complementary sequence is exactly our odd--offset class
$\mathcal{K}$, enumerated in closed form by $k(n).$ The Connell's sequence splits as a disjoint union
$$
\mathcal{K}\ \cup\ \{1^2,2^2,3^2,\ldots\}.
$$
Equivalently, deleting the perfect squares from the Connell sequence (A001614) produces precisely the complementary Mendelsohn sequence (A195437), i.e.\ the sequence $k(n)$.

We derive closed forms for the sequences of Connell and Mendelsohn using our own work; which in turn connects the parity of $U(n)$ to the aforementioned sequences.

\begin{theorem}\label{thm:connell_mendelsohn}
The following holds in $\mathbf{C}[[x]]$:
\begin{enumerate}[label=(\roman*)]
\item \textbf{Mendelsohn.} One has
$$
\sum_{n=0}^{\infty}x^{m(n)} \;=\; J(x).
$$
In particular $m(n)=j(n)$ for all $n\ge0$, and hence
$$
m(n)=2n-\left\lfloor \frac{\sqrt{8n+1}-1}{2}\right\rfloor,
$$
which is equation ~\ref{eq:mendelsohn}.
\item \textbf{Connell.} For $n\ge 1$, 
$$
c(n)=m(n-1)+1=j(n-1)+1,
$$ 
and consequently
$$
c(n)=2n-\left\lfloor\frac{1+\sqrt{8n-7}}{2}\right\rfloor,
$$
which is equation ~\ref{eq:korsak}.
\end{enumerate}
\end{theorem}

\begin{proof}
For \textit{(i)} we have by Section~\ref{sec:zerosandones} that
$$
J(x)=\sum_{n\in\mathcal{J}}x^n.
$$
The terms of Mendelsohn's procedure are precisely the integers in $\mathcal{J}$ listed in increasing order (see \eqref{eq:J-def}--\eqref{eq:Jset}), hence
$$
\sum_{n=0}^{\infty}x^{m(n)}=J(x).
$$
Since $j(n)$ denotes the $(n+1)$-th smallest element of $\mathcal{J}$, we also have
$J(x)=\sum_{n=0}^{\infty}x^{j(n)}$, and therefore $m(n)=j(n)$ for all $n\ge0$.
Substituting the closed form for $j(\cdot)$ from Theorem~\ref{thm:main1}(a) yields ~\ref{eq:mendelsohn}. 

For the proof of \textit{(ii)} we begin by using the explicit form of $J(x)$ from Section~\ref{sec:zerosandones},
$$
xJ(x)=\sum_{\ell=0}^{\infty}\sum_{r=0}^{\ell}x^{\ell^2+(2r+1)}.
$$
Split off the last term $r=\ell$ in the inner sum:
$$
xJ(x)=\sum_{\ell=0}^{\infty}\sum_{r=0}^{\ell-1}x^{\ell^2+(2r+1)}\;+\;\sum_{\ell=0}^{\infty}x^{\ell^2+(2\ell+1)}.
$$
The first double sum is exactly the odd--offset series $K(x)$ (see the block description in \eqref{eq:K-def} and the union $\mathcal{K}=\bigcup_{\ell\ge1}K_\ell$),
and the second sum simplifies to
$$
\sum_{\ell=0}^{\infty}x^{(\ell+1)^2}=\sum_{\ell=1}^{\infty}x^{\ell^2}=F(x;1,0)-1.
$$
By part (i) we have
$$
J(x)=\sum_{n=0}^{\infty}x^{m(n)}=\sum_{n=0}^{\infty}x^{j(n)}.
$$
Multiplying by $x$ gives
$$
xJ(x)=\sum_{n=0}^{\infty}x^{m(n)+1}=\sum_{n=0}^{\infty}x^{j(n)+1}.
$$
Since $m(n)$ (equivalently $j(n)$) is strictly increasing, the exponents occurring in $xJ(x)$, read in increasing order, are exactly
$$
1+m(0),\ 1+m(1),\ 1+m(2),\ \dots,
$$
hence $c(n)=m(n-1)+1=j(n-1)+1$ for $n\ge1$.. Finally, substitute the closed form for $j(\cdot)$ from Theorem~\ref{thm:main1}(a):
$$
c(n)=1+2(n-1)-\left\lfloor\frac{\sqrt{8(n-1)+1}-1}{2}\right\rfloor
=2n-1-\left\lfloor\frac{\sqrt{8n-7}-1}{2}\right\rfloor.
$$
Using $\lfloor y\rfloor+1=\lfloor y+1\rfloor$, this rewrites as
$$
c(n)=2n-\left\lfloor\frac{1+\sqrt{8n-7}}{2}\right\rfloor,
$$
which is ~\ref{eq:korsak}.
\end{proof}

\Needspace{10\baselineskip}
\begin{corollary}[OEIS interpretation of the parity classes]\label{cor:oeis_parity}
For $n\ge 1$ the following hold:
\begin{enumerate}[label=(\roman*)]
\item If $n$ belongs to the sequence A133280 then $U(n)\equiv 1 \pmod{2}.$
\item If $n$ belongs to the sequence A195437 then $U(n)\equiv 0 \pmod{2}.$
\item If $n$ belongs to the sequence A001614 then $U(n)\equiv 1 \pmod{2}$ if and only if $n$ is a perfect square, otherwise $U(n)\equiv 0\pmod{2}$.
\end{enumerate}
\end{corollary}

\begin{proof}
(i) By Theorem~\ref{thm:connell_mendelsohn}(i), Mendelsohn's sequence enumerates $\mathcal{J}$ in increasing order (indeed $m(n)=j(n)$), hence $n\in\text{A133280}$ implies $n\in\mathcal{J}$.
By Lemma~\ref{lem:npar-noemum}, for $n\ge 1$ one has $n\in\mathcal{J}$ if and only if $U(n)$ is odd, i.e.\ $U(n)\equiv 1\pmod{2}$. (ii) By construction $\mathcal{J}$ and $\mathcal{K}$ partition $\mathbf{N}$, so $\mathcal{K}=\mathbf{N}\setminus\mathcal{J}$.
The sequence A195437 is defined as the complement of A133280, hence $n\in\text{A195437}$ implies $n\in\mathcal{K}$. Applying Lemma~\ref{lem:npar-noemum} again yields $U(n)$ even, i.e.\ $U(n)\equiv 0\pmod{2}$. (iii) By Theorem~\ref{thm:connell_mendelsohn}(ii), every Connell number has the form $ n=c(t)=j(t-1)+1$ for $t\ge 1,$ with $j(t-1)\in\mathcal{J}$. Write
$j(t-1)=m^2+2r$ with $0\le r\le m.$ Then $ n=j(t-1)+1=m^2+(2r+1).$ If $0\le r\le m-1$, this has odd offset in the $m$-block, so $n\in\mathcal{K}$ and Lemma~\ref{lem:npar-noemum} gives $U(n)\equiv 0\pmod{2}$. If $r=m$, then $n=m^2+(2m+1)=(m+1)^2$ is a perfect square, hence $n\in\mathcal{J}$ (offset $0$ in the $(m+1)$-block) and Lemma~\ref{lem:npar-noemum} gives $U(n)\equiv 1\pmod{2}$.
\end{proof}

\begin{remark}
Solutions to Connell's query were also submitted by Leonard Carlitz, George Glauberman, and Donald E.\ Knuth.
\end{remark}


\end{document}